\documentclass[a4paper,12pt]{article}
\usepackage{amssymb,amsmath,amsthm,times,mathrsfs,fullpage}
\usepackage[pdftex]{hyperref}
\hypersetup{plainpages=True, pdfstartview=FitH, bookmarksopen=true,
colorlinks=true,linkcolor=blue,citecolor=blue}
\usepackage{tikz,pgf}
\usepackage{epstopdf}
\usepackage{todonotes}
\usepackage[nospace,noadjust]{cite}
\allowdisplaybreaks

\begin{document}

\newtheorem{thm}{Theorem}[section]
\newtheorem{cor}[thm]{Corollary}
\newtheorem{lmm}[thm]{Lemma}
\newtheorem{conj}[thm]{Conjecture}
\newtheorem{pro}[thm]{Proposition}
\theoremstyle{definition}\newtheorem{df}[thm]{Definition}
\theoremstyle{remark}\newtheorem{rem}[thm]{Remark}

\newcommand{\C}{\mathcal C}
\newcommand{\R}{\mathcal R}
\newcommand{\N}{\mathcal N}

\title{A~Cube-Root~Phase~Transition~in~Tree-Child~Networks \\ and the Enumeration Threshold for Galled Networks}
\author{Michael Fuchs \\ Department of Mathematical Sciences \\ National Chengchi University \\ Taipei 116 \\ Taiwan
\and
Hao Yu\\
Department of Mathematics \\
National University of Singapore \\
Singapore 119076 \\
Republic of Singapore}
\date{\today}
 
\maketitle

\begin{abstract}
We prove two surprising results about phylogenetic networks. First, we show that the structure of tree-child networks with $n$ leaves and $k$ reticulation nodes undergoes a sharp phase transition at $n^{1/3}$: if $k=o(n^{1/3})$, then a random tree-child network is almost surely a semi-simplex tree-child network, whereas if $k/n^{1/3}\rightarrow\infty$ and $k=o(n^{1/2})$, it is almost surely not. Second, we show that this result implies that the asymptotic counting formula for galled networks with $n$ leaves and a fixed number $k$ of reticulation nodes remains valid in the range $k=o(n^{1/3})$, but not beyond. This is in strong contrast to recently established results for the asymptotic counting formulas for tree-child and normal networks with $n$ leaves and $k$ reticulation nodes, which are valid in the (optimal) range $k=o(n^{1/2})$.
\end{abstract}

\section{Introduction and Results}

Over the last few decades, phylogenetic networks have increasingly replaced phylogenetic trees as the standard models for many evolutionary scenarios in biology; see \cite{HuRuSc,St}. Consequently, a rich mathematical theory has emerged, with many recent advances. Some of these advances concern the enumeration and stochastic behavior of various network classes; see, e.g., \cite{BoGaMa,LiLiu3Xi,McSeWe,PoBa,St}. In this paper, we add two new results in this area. The literature most closely related to our results is reviewed below. 

We start with the precise definition of a (rooted, binary) phylogenetic network, which is the following graph-theoretic structure.

\begin{df}[Phylogenetic Networks]
A (rooted, binary) {\it phylogenetic network} is a simple directed acyclic graph (DAG) whose nodes can be classified as follows.
\begin{itemize}
\item A unique {\it root} of indegree $0$ and outdegree $1$;
\item {\it Leaves}, which are bijectively labeled nodes of indegree $1$ and outdegree $0$;
\item {\it Tree nodes}, which are nodes of indegree $1$ and outdegree $2$;
\item {\it Reticulation nodes}, which are nodes of indegree $2$ and outdegree $1$.
\end{itemize}
\end{df}

Reticulation nodes (or \emph{reticulations} for short) are the nodes that distinguish general phylogenetic networks from phylogenetic trees, which form a subclass of the former. The \emph{size} of a network is its number of leaves. Moreover, if a network has $n$ leaves, then these are labeled by labels from the set $\{1,\ldots,n\}$.

As the space of phylogenetic networks is huge, many subclasses have been defined to make it more tractable; see \cite{KoPoLuWi}. In this paper, we consider two of them. The first is the prominent class of tree-child networks together with two of its subclasses; see \cite{CaRoVa} where it was introduced.

\begin{df}[Tree-Child Networks]
A phylogenetic network is called a {\it tree-child network} if every non-leaf node has at least one child that is a tree node or a leaf. In addition, we call a tree-child network {\it simplex} if the child of every reticulation node is a leaf, and {\it semi-simplex} if the child of every reticulation node is the root of a subtree.
\end{df}

Let ${\rm TC}_{n,k}$ denote the number of tree-child networks with $n$ leaves and $k$ reticulations. The following asymptotic counting result was recently proved in \cite{YuZh}: as $n\rightarrow\infty$,
\begin{equation}\label{asymp-k-small}
{\rm TC}_{n,k}\sim\frac{2^{k-1}\sqrt{2}}{k!}\left(\frac{2}{e}\right)^n n^{n+2k-1}
\end{equation}
uniformly in the optimal range $k=o(n^{1/2})$. This result improved earlier results from \cite{FuGiMa1,FuGiMa2,FuHuYu}, where the above asymptotics was proved only for fixed $k$. In addition, \cite{YuZh} showed that (\ref{asymp-k-small}) also holds for {\it normal networks} (introduced in \cite{Wi}), which are tree-child networks without {\it shortcuts}, i.e., edges $(u,v)$ such that $v$ can also be reached from $u$ via a path of length $\geq 2$.

Next, for simplex tree-child networks with $n$ leaves and $k$ reticulations, whose number we denote by ${\rm STC}_{n,k}$, the following closed-form expression is known:
\begin{equation}\label{formula-STC}
{\rm STC}_{n,k}=\binom{n}{k}\frac{(2n-2)!}{2^{n-1}(n-k-1)!},\qquad (0\leq k\leq n-1);
\end{equation}
see \cite{CaZh}. Thus, by Stirling's formula, as $n\rightarrow\infty$,
\[
{\rm STC}_{n,k}\sim\frac{\sqrt{2}}{2k!}\left(\frac{2}{e}\right)^nn^{n+2k-1}
\]
again throughout the optimal range $k=o(n^{1/2})$.

As for the number ${\rm SSTC}_{n,k}$ of semi-simplex tree-child networks with $n$ leaves and $k$ reticulations, it was proved in \cite{FuHuYu} that, for fixed $k$ and as $n\rightarrow\infty$, a tree-child network with $n$ leaves and $k$ reticulations chosen uniformly at random is almost surely semi-simplex, i.e.,
\begin{equation}\label{asymp-SSTC-k-fixed}
{\rm SSTC}_{n,k}\sim\frac{2^{k-1}\sqrt{2}}{k!}\left(\frac{2}{e}\right)^n n^{n+2k-1},
\end{equation}
for fixed $k$ as $n\rightarrow\infty$. In light of the above results, one would expect that the asymptotics in (\ref{asymp-SSTC-k-fixed}) remains true in the range $k=o(n^{1/2})$. Surprisingly, this turns out to be false.

\begin{pro}\label{result-SSTC}
As $n\rightarrow\infty$ and $k=o(n^{1/2})$,
\[
{\rm SSTC}_{n,k}\sim\frac{2^{k-1}\sqrt{2}}{k!}\left(\frac{2}{e}\right)^nn^{n+2k-1}e^{-\sqrt{2k^3/n}}.
\]
Consequently, (\ref{asymp-SSTC-k-fixed}) continues to hold in the optimal range $k=o(n^{1/3})$.
\end{pro}

From this result, together with (\ref{asymp-k-small}) from \cite{YuZh}, we obtain the following sharp structural phase transition result for tree-child networks, which is our first main result. To the best of our knowledge, this is the first such result in combinatorial phylogenetics, whereas phase transition results are a central theme in probabilistic combinatorics; see, e.g., \cite{Bo,ErRe,JaLuRu}.

\begin{thm}
Let $E_{n,k}$ denote the event that a randomly chosen tree-child network with $n$ leaves and $k$ reticulations is semi-simplex. Then, as $n\rightarrow\infty$ and $k=o(n^{1/2})$,
\[
{\mathbb P}(E_{n,k})\rightarrow\begin{cases} 1,&\ \text{if}\ k=o(n^{1/3}); \\ e^{-\sqrt{2}c^{3/2}},&\ \text{if}\ k\sim cn^{1/3}; \\ 0, &\ \text{if}\ k/n^{1/3}\rightarrow\infty.\end{cases}
\]
\end{thm}

Another consequence of Proposition~\ref{result-SSTC} concerns the counting of galled networks, another important network class in phylogenetics; see \cite{HuKl}.

\begin{df}
A phylogenetic network is called a {\it galled network} if every reticulation is in a {\it tree cycle}, i.e., a union of two edge-disjoint paths from a common tree node to a common reticulation with all other nodes being tree nodes.
\end{df}

Denote by ${\rm GN}_{n,k}$ the number of galled networks with $n$ leaves and $k$ reticulations. It was proved in \cite{ChFu} that for $k$ fixed and as $n\rightarrow\infty$,
\begin{equation}\label{asymp-GN-k-small}
{\rm GN}_{n,k}\sim {\rm TC}_{n,k}\sim\frac{2^{k-1}\sqrt{2}}{k!}\left(\frac{2}{e}\right)^n n^{n+2k-1};
\end{equation}
see also \cite{ChFuYu}. One might again guess that this extends to the range $k=o(n^{1/2})$, however, this again turns out to be wrong.

\begin{thm}\label{result-GN}
As $n\rightarrow\infty$ and $k=o(n^{1/2})$,
\[
{\rm GN}_{n,k}=\frac{2^{k-1}\sqrt{2}}{k!}\left(\frac{2}{e}\right)^nn^{n+2k-1}\left(e^{-\sqrt{2k^3/n}}+o(1)\right).
\]
\end{thm}

Thus, the asymptotic result (\ref{asymp-GN-k-small}) remains valid precisely in the range $k=o(n^{1/3})$. This resolves the open problem posed in the conclusion of \cite{YuZh}. Also, beyond that range (but still within the range $k=o(n^{1/2})$), the ratio of galled networks to tree-child networks tends to $0$. This is consistent with the fact that, for a fixed number $n$ of leaves, tree-child networks are asymptotically much more numerous than galled networks; see the main results~of~\cite{FuYuZh1}~and~\cite{FuYuZh2}.

We conclude the introduction with a brief outline of the paper. In the next section, we prove Proposition~\ref{result-SSTC}. The proof of Theorem~\ref{result-GN} is given in Section~\ref{gn}. We end the paper with some final remarks in Section~\ref{con}.

\section{Semi-simplex Tree-Child Networks}

In this section, we prove Proposition~\ref{result-SSTC}. We start with the following exponential generating function, which is motivated by (\ref{formula-STC}):
\[
F_k(z):=\sum_{\ell\geq 1}\frac{(2\ell+2k-2)!}{2^{\ell+k-1}(\ell-1)!}\cdot\frac{z^{\ell}}{\ell!}.
\]
Note that this is the exponential generating function of the number of simplex tree-child networks with $\ell+k$ leaves and $k$ reticulations, where the leaves below the reticulations receive labels from the set $\{1,\ldots,k\}$. Thus, when $k=0$, $F_0(z)$ counts phylogenetic trees, whose exponential generating function is given by
\[
T(z):=F_0(z)=1-\sqrt{1-2z};
\]
see, e.g., page 129 in \cite{FlSe}.

The exponential generating function for the number of semi-simplex tree-child networks is derived from $F_k(z)$ as follows.

\begin{lmm}\label{EGF-SSTC}
The exponential generating function of the number of semi-simplex tree-child networks with $n$ leaves and $k$ reticulations is given by
\[
F_k(z)\frac{T(z)^k}{k!}=F_k(z)\frac{(1-\sqrt{1-2z})^k}{k!}.
\]
\end{lmm}
\begin{proof}
By definition, a semi-simplex tree-child network with $n$ leaves and $k$ reticulations is obtained from a simplex tree-child network with $\ell+k$ leaves and $k$ reticulations, where the leaves below the reticulations are labeled by $1,\ldots,k$, by replacing the leaves below the reticulations with an unordered forest of phylogenetic trees whose leaf labels consist of the set $\{1,\ldots,n-\ell\}$. The replacement is performed as follows. The leaf with label $1$ is replaced by the phylogenetic tree in the forest with the smallest label, the leaf with label $2$ is replaced by the phylogenetic tree with the second smallest label, and so on. Finally, the leaves are (suitably) relabeled so that their labels are from the set $\{1,\ldots,n\}$. Using exponential generating functions, this construction translates directly into the claimed expression.
\end{proof}

We next need a result for $F_k(z)$.
\begin{lmm}\label{exp-Fkz}
For $k\geq 1$,
\[
F_k(z)=\frac{(2k)!}{2^k}(1-2z)^{-2k+1/2}p_k(z),
\]
where $p_k(z)$ is the polynomial of degree $k$ given by
\begin{equation}\label{pkz}
p_k(z):=\sum_{\ell=1}^{k}\frac{1}{2^{\ell-1}\ell}\binom{2\ell-2}{\ell-1}\binom{2k-2}{2\ell-2}z^{\ell}.
\end{equation}
\end{lmm}
\begin{proof}
Using the definition of the Gauss hypergeometric function, 
\begin{equation}\label{Fk-hyper}
F_k(z)=\frac{(2k)!}{2^k}z{}_2F_1(k+1/2,k+1,2;2z).
\end{equation}
Now recall that
\[
_{2}F_1(k+1/2,k+1,2;x)=(1-x)^{-2k+1/2}{}_{2}F_1(1-k,3/2-k,2;x).
\]
Plugging this into (\ref{Fk-hyper}) gives the claimed expression with
\[
p_k(z):=z{}_2F_1(1-k,3/2-k,2;2z), 
\]
where (\ref{pkz}) is obtained by straightforward simplification of the coefficient of the Gauss hypergeometric function with the above parameters. (The finite degree of the series comes from the fact that the first parameter is an integer.)
\end{proof}
\begin{rem}
For later use, we remark that $p_k(1/2)$ simplifies to:
\begin{equation}\label{pk1/2}
p_k(1/2)=\sum_{\ell=1}^{k}\frac{1}{2^{2\ell-1}\ell}\binom{2\ell-2}{\ell-1}\binom{2k-2}{2\ell-2}=\frac{1}{2^{2k}k}\binom{4k-2}{2k-1}.
\end{equation}
This identity follows either by a direct generating function argument or from standard properties of the Gauss hypergeometric function.
\end{rem}

Before we can prove Proposition~\ref{result-SSTC} from the exponential generating function in Lemma~\ref{EGF-SSTC}, which is completely explicit due to Lemma~\ref{exp-Fkz}, we need a final technical lemma. We formulate this lemma in a more general form, as it will also be used in the proof of Lemma~\ref{rel-GTC-SSTC} in the next section.

We recall that $[z^n]f(z)$ denotes the $n$-th Maclaurin coefficient of $f(z)$.

\begin{lmm}\label{tech-lmm}
Let $1\leq\alpha=\Theta(k)$ and $0\leq m\leq k$. Then, as $n\rightarrow\infty$,
\[
[z^n](1-z)^{-\alpha}(1-\sqrt{1-z})^{m}=\frac{n^{\alpha-1}}{\Gamma(\alpha)}e^{-m\sqrt{\alpha/n}}\left(1+{\mathcal O}\left(\sqrt{\frac{k}{n}}+\frac{k^2}{n}\right)\right),
\]
uniformly for $1\leq k=o(n^{1/2})$.
\end{lmm}
\begin{proof}
The coefficient in question can be written as a convolution
\begin{equation}\label{conv}
[z^n](1-z)^{-\alpha}(1-\sqrt{1-z})^{m}=\sum_{j=0}^{n}[z^j]\tilde{T}(z)^m\binom{n-j+\alpha-1}{n-j},
\end{equation}
where $\tilde{T}(z)=T(z/2)=1-\sqrt{1-z}$. 

We first consider the binomial coefficient which can be written with the help of the gamma function as follows:
\[
\binom{n-j+\alpha-1}{n-j}=\frac{\Gamma(n-j+\alpha)}{\Gamma(\alpha)\Gamma(n-j+1)}.
\]
Our goal is to prove for it the following uniform asymptotic expansion: for $0\leq j\leq n$, as $n\rightarrow\infty$,
\begin{equation}\label{unif-gamma-ratio}
\frac{\Gamma(n-j+\alpha)}{\Gamma(n-j+1)}=n^{\alpha-1}e^{-\alpha j/n}\left(1+{\mathcal O}\left(\frac{\alpha j^2}{n^2}+\frac{\alpha^2}{n}+\frac{j}{n}\right)\right).
\end{equation}
We split the proof into two cases.

First, if $0\leq j\leq n/2$, by the well-known asymptotics of the ratio of two gamma functions
\begin{equation}\label{gamma-ratio}
\frac{\Gamma(n-j+\alpha)}{\Gamma(n-j+1)}=(n-j)^{\alpha-1}\left(1+{\mathcal O}\left(\frac{\alpha^2}{n-j}\right)\right)=(n-j)^{\alpha-1}\left(1+{\mathcal O}\left(\frac{\alpha^2}{n}\right)\right),
\end{equation}
where the last step follows since $n-j\geq n/2$. Next, for the leading term, as $n\rightarrow\infty$,
\begin{align*}
(n-j)^{\alpha-1}&=n^{\alpha-1}e^{(\alpha-1)\log(1-j/n)}\\[0.15cm]
&=n^{\alpha-1}e^{-(\alpha-1)j/n+{\mathcal O}(\alpha j^2/n^2)}\\
&=n^{\alpha-1}e^{-(\alpha-1)j/n}\left(1+{\mathcal O}\left(\frac{\alpha j^2}{n^2}\right)\right)\\
&=n^{\alpha-1}e^{-\alpha j/n}\left(1+{\mathcal O}\left(\frac{\alpha j^2}{n^2}+\frac{j}{n}\right)\right).
\end{align*}
Plugging this into (\ref{gamma-ratio}) yields (\ref{unif-gamma-ratio}) in this case.

Next, we consider the range $n/2<j\leq n$. Here, we define
\begin{equation}\label{factor-2}
R(j;n)=\frac{\Gamma(n-j+\alpha)}{\Gamma(n-j+1)}\cdot\frac{\Gamma(n+1)}{\Gamma(n+\alpha)}\cdot\left(\frac{n+\alpha}{n}\right)^j.
\end{equation}
Note that the second term in $R(j;n)$ has the expansion
\begin{equation}\label{factor-3}
\frac{\Gamma(n+1)}{\Gamma(n+\alpha)}=n^{1-\alpha}\left(1+{\mathcal O}\left(\frac{\alpha^2}{n}\right)\right)
\end{equation}
and the third term has the expansion
\[
\left(\frac{n+\alpha}{n}\right)^j=e^{\alpha j/n}\left(1+{\mathcal O}\left(\frac{\alpha^2 j}{n^2}\right)\right)=e^{\alpha j/n}\left(1+{\mathcal O}\left(\frac{\alpha^2}{n}\right)\right)
\]
which correspond to the reciprocals of the main terms on the right-hand side of (\ref{gamma-ratio}). (That is why these terms were added as multiplicative factors.) The sequence $R(j;n)$ has the product representation
\[
R(j;n)=\prod_{\ell=0}^{j-1}\left(1+\frac{n-\alpha\ell}{n(n-\ell+\alpha-1)}\right)
\]
whose terms exceed one if and only if $n-\alpha\ell>0$ which in turn holds if and only if $\ell<n/\alpha$. In that case, since $\alpha\geq 1$, we have
\[
\frac{n-\alpha\ell}{n-\ell+\alpha-1}\leq\alpha
\]
and thus these terms are at most $1+\alpha/n$. In addition, there are at most $n/\alpha+1$ such terms. Consequently, 
\[
R(j;n)\leq\left(1+\frac{\alpha}{n}\right)^{n/\alpha+1}={\mathcal O}(1)
\]
for $n/2<j\leq n$. Combining this with (\ref{factor-2}) and (\ref{factor-3}) gives
\[
\frac{\Gamma(n-j+\alpha)}{\Gamma(n-j+1)}={\mathcal O}\left(n^{\alpha-1}e^{-\alpha j/n}\right)
\]
for $n/2<j\leq n$. From this (\ref{unif-gamma-ratio}) follows also for this range due to the appearance of $j/n$ in the ${\mathcal O}$-term of (\ref{unif-gamma-ratio}) which is $j/n>1/2$.

We now plug (\ref{unif-gamma-ratio}) into (\ref{conv}) which produces the following sums
\begin{equation}\label{three-sums}
\sum_{j=0}^{n}[z^j]\tilde{T}(z)^m\rho^j,\qquad \sum_{j=0}^{n}j[z^j]\tilde{T}(z)^m\rho^j,\qquad \sum_{j=0}^{n}j^2[z^j]\tilde{T}(z)^m\rho^j,
\end{equation}
where $\rho=e^{-\alpha/n}$ with $\rho\sim1$ as $n\rightarrow\infty$. These sums can be estimated by using generating functions because
\[
\sum_{j=0}^{\infty}[z^j]\tilde{T}(z)^m\rho^jz^j=\tilde{T}(\rho z)^m
\]
and thus
\begin{equation}\label{first-mom}
\sum_{j=0}^{\infty}j[z^j]\tilde{T}(z)^m\rho^jz^{j-1}=\frac{d}{dz}\tilde{T}(\rho z)^m=m\tilde{T}(\rho z)^{m-1}\tilde{T}'(\rho z)\rho
\end{equation}
and
\begin{align}
\sum_{j=0}^{\infty}j^2[z^j]\tilde{T}(z)^m\rho^jz^{j-1}&=\frac{d}{dz}\left(z\frac{d}{dz}\tilde{T}(\rho z)^m\right)\nonumber\\
&=m\tilde{T}(\rho z)^{m-1}\tilde{T}'(\rho z)\rho+m(m-1)z\tilde{T}(\rho z)^{m-2}\tilde{T}'(\rho z)^2\rho^2\nonumber\\
&\qquad+mz\tilde{T}(\rho z)^{m-1}\tilde{T}''(\rho z)\rho^2.\label{third-mom}
\end{align}

We start by deriving an asymptotic expansion of the first sum in (\ref{three-sums}):
\begin{equation}\label{first-sum-dec}
\sum_{j=0}^{n}[z^j]\tilde{T}(z)^m\rho^j=\tilde{T}(\rho)^m-\sum_{j>n}[z^j]\tilde{T}(z)^m\rho^j.
\end{equation}
Note that
\begin{align}
\tilde{T}(\rho)^m&=(1-\sqrt{1-\rho})^m=e^{m\log(1-\sqrt{1-\rho})}=e^{-m\sqrt{1-\rho}+{\mathcal O}(m\alpha/n)}\nonumber\\
&=e^{-m\sqrt{\alpha/n}+{\mathcal O}(m\alpha/n)}=e^{-m\sqrt{\alpha/n}}\left(1+{\mathcal O}\left(\frac{k^2}{n}\right)\right)\label{tilde-T}
\end{align}
where in the last step, we have used that $\alpha=\Theta(k)$ and $m\leq k$. Moreover,
\[
\sum_{j>n}[z^j]\tilde{T}(z)^m\rho^j\leq\frac{1}{n}\sum_{j=0}^{\infty}j[z^j]\tilde{T}(z)^m\rho^j=\frac{m\tilde{T}(\rho)^{m-1}\tilde{T}'(\rho)\rho}{n},
\]
where the series is obtained by evaluating (\ref{first-mom}) at $z=1$. Note that $\tilde{T}'(\rho)=(1-\rho)^{-1/2}/2\sim\sqrt{n}/(2\sqrt{\alpha})$
and $\tilde{T}(\rho)\sim\rho\sim1$ as $n\rightarrow\infty$. Thus,
\begin{equation}\label{second-ub}
m\tilde{T}(\rho)^{m-1}\tilde{T}'(\rho)\rho={\mathcal O}\left(\tilde{T}(\rho)^m\frac{m\sqrt{n}}{\sqrt{\alpha}}\right)={\mathcal O}\left(e^{-m\sqrt{\alpha/n}}\sqrt{kn}\right),
\end{equation}
where we have again used that $\alpha=\Theta(k)$ and $m\leq k$. Consequently,
\[
\sum_{j>n}[z^j]\tilde{T}(z)^m\rho^j={\mathcal O}\left(e^{-m\sqrt{\alpha/n}}\sqrt{\frac{k}{n}}\right).
\]
Combining this with (\ref{tilde-T}) and plugging into (\ref{first-sum-dec}) gives
\begin{equation}\label{first-sum}
\sum_{j=0}^{n}[z^j]\tilde{T}(z)^m\rho^j=e^{-m\sqrt{\alpha/n}}\left(1+{\mathcal O}\left(\sqrt{\frac{k}{n}}+\frac{k^2}{n}\right)\right).
\end{equation}

Next, for the second sum in (\ref{three-sums}):
\begin{equation}\label{second-sum}
\sum_{j=0}^{n}j[z^j]\tilde{T}(z)^m\rho^j\leq m\tilde{T}(\rho)^{m-1}\tilde{T}'(\rho)\rho={\mathcal O}\left(e^{-m\sqrt{\alpha/n}}\sqrt{kn}\right),
\end{equation}
where the last step follows from (\ref{second-ub}).

Finally, for the third sum in (\ref{three-sums}), we upper bound it by (\ref{third-mom}) evaluated at $z=1$ which gives
\[
\sum_{j=0}^{n}j^2[z^j]\tilde{T}(z)^m\rho^j={\mathcal O}\left(\tilde{T}(\rho)^m\left(\frac{m\sqrt{n}}{\sqrt{\alpha}}+\frac{m^2n}{\alpha}+\frac{mn^{3/2}}{\alpha^{3/2}}\right)\right),
\]
where we have used that $\tilde{T}''(\rho)=(1-\rho)^{-3/2}/4\sim n^{3/2}/(4\alpha^{3/2})$ as $n\rightarrow\infty$. Thus, since $\alpha=\Theta(k)$ and $m\leq k$,
\begin{equation}\label{third-sum}
\sum_{j=0}^{n}j^2[z^j]\tilde{T}(z)^m\rho^j={\mathcal O}\left(e^{-m\sqrt{\alpha/n}}\left(kn+n\sqrt{\frac{n}{k}}\right)\right).
\end{equation}

Now, we finally can plug (\ref{unif-gamma-ratio}) into (\ref{conv}) and use (\ref{first-sum}), (\ref{second-sum}), and (\ref{third-sum}) which gives the claimed result.
\end{proof}

\begin{rem}
The last lemma is a uniform version of the transfer theorems from singularity analysis; see Chapter VI in \cite{FlSe}. It could also have been derived using complex analysis; however, we decided to follow the above more elementary approach.
\end{rem}

As a corollary, we have the following.
\begin{cor}\label{tech-cor}
Let $q_k(z)$ be a polynomial of degree $k$ with non-negative coefficients. Then, with the same assumptions as in Lemma~\ref{tech-lmm},
\[
[z^n](1-z)^{-\alpha}(1-\sqrt{1-z})^{m}q_k(z)=q_k(1)\frac{n^{\alpha-1}}{\Gamma(\alpha)}e^{-m\sqrt{\alpha/n}}\left(1+{\mathcal O}\left(\sqrt{\frac{k}{n}}+\frac{k^2}{n}\right)\right),
\]
uniformly for $1\leq k=o(n^{1/2})$.
\end{cor}
\begin{proof}
Set $q_k(z)=\sum_{d=0}^{k}\omega_dz^d$. Then,
\begin{equation}\label{coeff-ext}
[z^n](1-z)^{-\alpha}(1-\sqrt{1-z})^mq_k(z)=\sum_{d=0}^{k}\omega_d[z^{n-d}](1-z)^{-\alpha}(1-\sqrt{1-z})^m.
\end{equation}
From the last lemma, uniformly in $0\leq d\leq k$,
\[
[z^{n-d}](1-z)^{-\alpha}(1-\sqrt{1-z})^m=\frac{(n-d)^{\alpha-1}}{\Gamma(\alpha)}e^{-m\sqrt{\alpha/(n-d)}}\left(1+{\mathcal O}\left(\sqrt{\frac{k}{n}}+\frac{k^2}{n}\right)\right).
\]
Note that
\[
(n-d)^{\alpha-1}=n^{\alpha-1}e^{(\alpha-1)\log(1-d/n)}=n^{\alpha-1}e^{{\mathcal O}(\alpha d/n)}=n^{\alpha-1}\left(1+{\mathcal O}\left(\frac{k^2}{n}\right)\right)
\]
and
\[
e^{-m\sqrt{\alpha/(n-d)}}=e^{-m\sqrt{\alpha/n}+{\mathcal O}(md\sqrt{\alpha}/n^{3/2})}=e^{-m\sqrt{\alpha/n}}\left(1+{\mathcal O}\left(\frac{k^{5/2}}{n^{3/2}}\right)\right).
\]
Thus, uniformly in $0\leq d\leq k$
\[
[z^{n-d}](1-z)^{-\alpha}(1-\sqrt{1-z})^m=\frac{n^{\alpha-1}}{\Gamma(\alpha)}e^{-m\sqrt{\alpha/n}}\left(1+{\mathcal O}\left(\sqrt{\frac{k}{n}}+\frac{k^2}{n}\right)\right).
\]
Plugging this into (\ref{coeff-ext}) gives the claimed result.
\end{proof}

Now, we can prove Proposition~\ref{result-SSTC} in a slightly stronger form.
\begin{proof}[Proof of Proposition~\ref{result-SSTC}]
From Lemma~\ref{EGF-SSTC},
\[
{\rm SSTC}_{n,k}=n![z^n]F_k(z)\frac{(1-\sqrt{1-2z})^k}{k!}.
\]
Plugging into this the expression for $F_k(z)$ from Lemma~\ref{exp-Fkz},
\begin{align*}
{\rm SSTC}_{n,k}&=\frac{(2k)!}{2^kk!}n![z^n](1-2z)^{-2k+1/2}(1-\sqrt{1-2z})^kp_k(z)\\
&=\frac{(2k)!}{2^kk!}2^nn![z^n](1-z)^{-2k+1/2}(1-\sqrt{1-z})^kp_k(z/2).
\end{align*}
Applying Corollary~\ref{tech-cor} (with $\alpha=2k-1/2$ and $m=k$) and Stirling's formula,
\begin{align*}
{\rm SSTC}_{n,k}&=\frac{(2k)!}{2^kk!}p_k(1/2)\frac{n^{2k-3/2}}{\Gamma(2k-1/2)}2^nn!e^{-\sqrt{2k^3/n}}\left(1+{\mathcal O}\left(\sqrt{\frac{k}{n}}+\frac{k^2}{n}\right)\right)\\
&=c_k\left(\frac{2}{e}\right)^nn^{n+2k-1}e^{-\sqrt{2k^3/n}}\left(1+{\mathcal O}\left(\sqrt{\frac{k}{n}}+\frac{k^2}{n}\right)\right),
\end{align*}
where
\begin{equation}\label{ck}
c_k:=\frac{(2k)!p_k(1/2)\sqrt{2\pi}}{2^kk!\Gamma(2k-1/2)}.
\end{equation}

What is left is to simplify the constant. From the duplication formula of the gamma function:
\[
\Gamma(2k-1/2)=\frac{(4k-3)!\sqrt{\pi}}{2^{4k-3}(2k-2)!}
\]
Plugging this into (\ref{ck}) and using (\ref{pk1/2}) gives
\[
c_k=\frac{2^{k-1}\sqrt{2}}{k!}
\]
which is the claimed result that in addition also contains an error term.
\end{proof}

\section{Galled Networks}\label{gn}

In this section, we prove Theorem~\ref{result-GN}, which is our second main result. The main step of the proof is to show that Proposition~\ref{result-SSTC} extends to networks that are both tree-child and galled. These networks have been named {\it galled tree-child networks} in \cite{ChFuYu}. Let ${\rm GTC}_{n,k}$ denote their number with $n$ leaves and $k$ reticulations. A useful feature of this class is that their bivariate exponential generating function, i.e.,
\[
G(z,v):=\sum_{n\geq 1}\sum_{k\geq 0}{\rm GTC}_{n,k}\frac{z^n}{n!}v^k
\]
satisfies an implicit equation involving the function $F_k(z)$ from the last section; see \cite{ChFuYu}.

\begin{lmm}[\cite{ChFuYu}]
We have,
\begin{equation}\label{eq-Gzv}
G(z,v)=\sum_{j\geq 0}F_j(z)\frac{(vG(z,v))^j}{j!}.
\end{equation}
\end{lmm}

Using the Lagrange inversion formula, a closed form expression can be derived from this for the exponential generating function of ${\rm GTC}_{n,k}$:
\[
E_k(z):=\sum_{n\geq 1}{\rm GTC}_{n,k}\frac{z^n}{n!}.
\]

\begin{lmm} For $k\geq 1$,
\begin{equation}\label{Ekz}
E_k(z)=\frac{1}{k+1}\sum_{s=1}^{k}\binom{k+1}{s}F_0(z)^{k+1-s}\sum_{\substack{j_1+\cdots+j_s=k\\j_i\geq1}}
\prod_{i=1}^{s}\frac{F_{j_i}(z)}{j_i!}.
\end{equation}
\end{lmm}
\begin{proof}
Set $H(z,v):=vG(z,v)$. Then, (\ref{eq-Gzv}) can be rewritten as:
\[
H(z,v)=v\Phi(H(z,v)),\qquad\Phi(\omega)=\sum_{j\geq 0}F_j(z)\frac{\omega^j}{j!}.
\]
This is exactly the setting to which the Lagrange inversion formula can be applied, which yields
\[
E_k(z)=[v^{k}]G(z,v)=[v^{k+1}]H(z,v)=\frac{1}{k+1}[\omega^k]\left(\sum_{j\geq 0}F_j(z)\frac{\omega^j}{j!}\right)^{k+1}.
\]
Expanding the power gives the claimed result.
\end{proof}

We will use this in order to show that ${\rm GTC}_{n,k}$ is closely related to ${\rm SSTC}_{n,k}$.

\begin{lmm}\label{rel-GTC-SSTC}
As $n\rightarrow\infty$,
\[
{\rm GTC}_{n,k}={\rm SSTC}_{n,k}\left(1+{\mathcal O}\left(\sqrt{\frac{k}{n}}\right)\right).
\]
uniformly for $k=o(n^{1/2})$.
\end{lmm}

The proof requires a final (technical) lemma. Set $a_k:=(2k)!/(2^k k!)$ and
\begin{equation}\label{asymp-bk}
b_k:=a_kp_k(1/2)=\frac{(4k-2)!}{2^{3k-1}(2k-1)!k!}b_k\sim\frac{8^k k!}{4\sqrt{2}\pi k^2}=\Theta\left(\frac{8^k k!}{k^2}\right),\qquad (k\rightarrow\infty),
\end{equation}
where the polynomials $p_k(z)$ are from Lemma~\ref{exp-Fkz}, the closed-form expression follows from identity (\ref{pk1/2}), and the asymptotic approximations follow from Stirling's formula. Then, we have the following property.
\begin{lmm}\label{final-lmm}
There exists an absolute constant $c_0>0$ such that for all $1\leq s\leq k$:
\[
\sum_{\substack{j_1+\cdots+j_s=k\\j_i\geq1}}b_{j_1}\cdots b_{j_s}\leq b_k\left(\frac{c_0}{k}\right)^{s-1}s!.
\]
\end{lmm}

\begin{proof}
We use induction on $s$. First, if $s=1$, the claim trivially holds. So, assume that the claim is true for $s-1$. We are going to show it for $s$. By using the induction hypothesis,
\begin{align*}
\sum_{\substack{j_1+\cdots+j_s=k\\j_i\geq1}}b_{j_1}\cdots b_{j_s}&=\sum_{r=1}^{k-s+1}b_r\sum_{\substack{j_2+\cdots+j_s=k-r\\j_i\geq1}}b_{j_2}\cdots b_{j_s}\\
&\leq c_0^{s-2}(s-1)!\sum_{r=1}^{k-s+1}\frac{b_rb_{k-r}}{(k-r)^{s-2}}.
\end{align*}
The claim follows from this, if we can prove that
\begin{equation}\label{claim}
\sum_{r=1}^{k-s+1}\frac{b_rb_{k-r}}{(k-r)^{s-2}}\leq c_0\frac{b_k}{k^{s-1}}
\end{equation}
for all $s\leq k$. In order to show this, we use (\ref{asymp-bk}) which gives
\begin{equation}\label{conv-sum}
\sum_{r=1}^{k-s+1}\frac{b_rb_{k-r}}{(k-r)^{s-2}}={\mathcal O}\left(8^k\sum_{r=1}^{k-s+1}\frac{r!(k-r)!}{r^2(k-r)^s}\right),
\end{equation}
where the implied constant is absolute. We break the sum inside the ${\mathcal O}$-term into three parts:
\begin{equation}\label{three-parts}
\sum_{r=1}^{k-s+1}\frac{r!(k-r)!}{r^2(k-r)^s}=\frac{(k-1)!}{(k-1)^s}+\sum_{2\leq r\leq k/2}\frac{r!(k-r)!}{r^2(k-r)^s}+\sum_{k/2<r\leq k-s+1}\frac{r!(k-r)!}{r^2(k-r)^s}.
\end{equation}

For the first part, we have
\[
\frac{(k-1)!}{(k-1)^s}=\frac{k!}{k^{s+1}}\cdot\frac{k^{s}}{(k-1)^s}\leq\frac{k!}{k^{s+1}}e^{s/(k-1)}={\mathcal O}\left(\frac{k!}{k^{s+1}}\right),
\]
where the implied constant is absolute. For the second part, by Stirling's formula,
\begin{align*}
\sum_{2\leq r\leq k/2}\frac{r!(k-r)!}{r^2(k-r)^s}&={\mathcal O}\left(\frac{1}{e^k}\sum_{2\leq r\leq k/2}r^{r-3/2}(k-r)^{k-r+1/2-s}\right)\\
&={\mathcal O}\left(\frac{1}{e^k\sqrt{k}}\sum_{2\leq r\leq k/2}r^{r-3/2}(k-r)^{k-r+1-s}\right)\\
&={\mathcal O}\left(\frac{k^{k-1/2-s}}{e^k\sqrt{k}}\sum_{2\leq r\leq k/2}\left(\frac{r}{k}\right)^{r-3/2}\left(1-\frac{r}{k}\right)^{k-r+1-s}\right)\\
&={\mathcal O}\left(\frac{k^k}{e^k k^{s+1}}\sum_{2\leq r\leq k/2}\left(\frac{r}{k}\right)^{r-3/2}\right)={\mathcal O}\left(\frac{k^k}{e^kk^{s+1}}\right)={\mathcal O}\left(\frac{k!}{k^{s+1}}\right),
\end{align*}
where from the second-last line to the last line, we used that $r\leq k-s+1$ (otherwise the sum in (\ref{claim}) is empty), and the implied constant is again absolute. Finally, for the last part in (\ref{three-parts}):
\[
\sum_{k/2<r\leq k-s+1}\frac{r!(k-r)!}{r^2(k-r)^s}={\mathcal O}\left(\frac{(k-s+1)!(s-1)!}{k^2}\sum_{r\geq s-1}\frac{1}{r^s}\right)={\mathcal O}\left(\frac{(k-s+1)!(s-1)!}{k^2(s-1)^{s-1}}\right),
\]
where we used that $r!(k-r)!$ is non-decreasing in the range $k/2<r\leq k-s+1$ and the last step follows by a standard integral approximation. Also, note that the implied constant is again absolute. Now, from the simple bound
\[
\binom{k}{s-1}\geq\left(\frac{k}{s-1}\right)^{s-1},
\]
we obtain that
\[
\frac{(k-s+1)!(s-1)!}{(s-1)^{s-1}}\leq\frac{k!}{k^{s-1}}.
\]
Thus,
\[
\sum_{k/2<r\leq k-s+1}\frac{r!(k-r)!}{r^2(k-r)^s}={\mathcal O}\left(\frac{k!}{k^{s+1}}\right)
\]
with the implied constant absolute.

Consequently, all three parts in (\ref{three-parts}) admit the same upper bound. Applying them to (\ref{conv-sum}) gives
\[
\sum_{r=1}^{k-s+1}\frac{b_rb_{k-r}}{(k-r)^{s-2}}={\mathcal O}\left(\frac{8^kk!}{k^2k^{s-1}}\right)={\mathcal O}\left(\frac{b_k}{k^{s-1}}\right),
\]
where the last step follows from (\ref{asymp-bk}). Since the implied constant is absolute, we have proved (\ref{claim}) which concludes the proof of the lemma.
\end{proof}

With the help of the last lemma, we can now prove Lemma~\ref{rel-GTC-SSTC}.

\begin{proof}[Proof of Lemma~\ref{rel-GTC-SSTC}] Note that for $s=1$, the right hand side of (\ref{Ekz}) is the exponential generating function of the number of semi-simplex tree-child networks with $n$ leaves and $k$ reticulations; see Lemma~\ref{EGF-SSTC}. Thus,
\begin{equation}\label{exp-GTC-SSTC}
{\rm GTC}_{n,k}=n![z^n]E_k(z)={\rm SSTC}_{n,k}+n![z^n]R_k(z),
\end{equation}
where
\[
R_k(z):=\frac{1}{k+1}\sum_{s=2}^k\binom{k+1}{s}F_0(z)^{k+1-s}\sum_{\substack{j_1+\cdots+j_s=k\\j_i\geq1}}
\prod_{i=1}^{s}\frac{F_{j_i}(z)}{j_i!}.
\]
Plugging the expression for $F_k(z)$ from Lemma~\ref{exp-Fkz} into this gives
\[
R_k(z)=\frac{1}{k+1}\sum_{s=2}^{k}\binom{k+1}{s}(1-2z)^{-2k+s/2}T(z)^{k+1-s}\underbrace{\sum_{\substack{j_1+\cdots+j_s=k\\j_i\geq1}}
\prod_{i=1}^{s}a_{j_i}p_{j_i}(z)}_{\displaystyle =:r_{k,s}(z)},
\]
where $T(z)=1-\sqrt{1-2z}$. Note that $r_{k,s}(z)$ is a polynomial of degree $k$ with non-negative coefficients. Thus, by Corollary~\ref{tech-cor},
\begin{align}
[z^n]R_k(z)&=\frac{1}{k+1}\sum_{s=2}^{k}\binom{k+1}{s}[z^n](1-2z)^{-2k+s/2}T(z)^{k+1-s}r_{k,s}(z)\nonumber\\
&=\frac{2^n}{k+1}\sum_{s=2}^{k}\binom{k+1}{s}[z^n](1-z)^{-2k+s/2}\tilde{T}(z)^{k+1-s}r_{k,s}(z/2)\nonumber\\
&={\mathcal O}\left(2^n
\sum_{s=2}^{k}\frac{k!r_{k,s}(1/2)}{(k+1-s)!s!}
\frac{n^{2k-s/2-1}}{\Gamma(2k-s/2)}e^{-(k+1-s)\sqrt{(2k-s/2)/n}}\right),\label{est-Rkz}
\end{align}
where $\tilde{T}(z)=T(z/2)$. Now, observe that by Lemma~\ref{final-lmm}, 
\[
r_{k,s}(1/2)=\sum_{\substack{j_1+\cdots+j_s=k\\j_i\geq1}}
\prod_{i=1}^{s}a_{j_i}p_{j_i}(1/2)=\sum_{\substack{j_1+\cdots+j_s=k\\j_i\geq1}}b_{j_1}\cdots b_{j_s}\leq b_k\left(\frac{c_0}{k}\right)^{s-1}s!.
\]
Moreover,
\[
\frac{\Gamma(2k-1/2)}{\Gamma(2k-s/2)}={\mathcal O}((2k)^{(s-1)/2})
\]
and
\[
e^{-(k+1-s)\sqrt{(2k-s/2)/n}}=e^{-k\sqrt{(2k-s/2)/n}+{\mathcal O}(s)}=e^{-\sqrt{2k^3/n}+{\mathcal O}(s)}.
\]
Plugging this into (\ref{est-Rkz}) gives
\[
[z^n]R_k(z)={\mathcal O}\left(\frac{b_k}{\Gamma(2k-1/2)}2^n n^{2k-3/2}e^{-\sqrt{2k^3/n}}\sum_{s=2}^k\frac{k!}{(k+1-s)!k^{s-1}}\left(\tilde{c}_0\frac{k}{n}\right)^{(s-1)/2}\right),
\]
where $\tilde{c}_0$ is a suitable constant. Note that
\[
\frac{b_k}{\Gamma(2k-1/2)}={\mathcal O}\left(\frac{2^k}{k!}\right)
\]
which follows by the same line of arguments as in the end of the proof of Proposition~\ref{result-SSTC}. (The above constant is up to the term $\sqrt{2\pi}$ equal to $c_k$.) Also,
\[
\sum_{s=2}^{k}\frac{k!}{(k+1-s)!k^{s-1}}\left(\tilde{c}_0\frac{k}{n}\right)^{(s-1)/2}\leq\sum_{s=2}^{k}\left(\tilde{c}_0\frac{k}{n}\right)^{(s-1)/2}={\mathcal O}\left(\sqrt{\frac{k}{n}}\right).
\]
Thus,
\[
[z^n]R_k(z)={\mathcal O}\left(\frac{2^k}{k!}2^n n^{2k-3/2}e^{-\sqrt{2k^3/n}}\sqrt{\frac{k}{n}}\right)
\]
and consequently, by Stirling's formula,
\[
n![z^n]R_k(z)={\mathcal O}\left(\frac{2^k}{k!}\left(\frac{2}{e}\right)^n n^{n+2k-1}e^{-\sqrt{2k^3/n}}\sqrt{\frac{k}{n}}\right).
\]
Plugging this into (\ref{exp-GTC-SSTC})  completes the proof of the lemma.
\end{proof}

Combining Lemma~\ref{rel-GTC-SSTC} with Proposition~\ref{result-SSTC}, we obtain the following consequence.
\begin{cor}\label{asymp-GTC}
As $n\rightarrow\infty$ and $k=o(n^{1/2})$,
\[
{\rm GTC}_{n,k}\sim\frac{2^{k-1}\sqrt{2}}{k!}\left(\frac{2}{e}\right)^n n^{n+2k-1}e^{-\sqrt{2k^3/n}}.
\]
\end{cor}

We can now deduce from this our second main result.

\begin{proof}[Proof of Theorem~\ref{result-GN}] Denote by ${\rm PN}_{n,k}$ the number of (general) phylogenetic networks with $n$ leaves and $k$ reticulations. It was proved in \cite{YuZh} that, as $n\rightarrow\infty$
\begin{equation}\label{asymp-PN}
{\rm PN}_{n,k}\sim\frac{2^{k-1}\sqrt{2}}{k!}\left(\frac{2}{e}\right)^n n^{n+2k-1}
\end{equation}
uniformly for $k=o(n^{1/2})$. Now observe that a network which is a galled network but not a galled tree-child network is contained in the set of phylogenetic networks which are not tree-child networks. Thus,
\[
0\leq {\rm GN}_{n,k}-{\rm GTC}_{n,k}\leq {\rm PN}_{n,k}-{\rm TC}_{n,k}.
\]
Using (\ref{asymp-k-small}) and (\ref{asymp-PN}) yields
\[
{\rm GN}_{n,k}-{\rm GTC}_{n,k}=o\left(\frac{2^k}{k!}\left(\frac{2}{e}\right)^nn^{n+2k-1}\right)
\]
and consequently, from Corollary~\ref{asymp-GTC},
\[
{\rm GN}_{n,k}=\frac{2^{k-1}\sqrt{2}}{k!}\left(\frac{2}{e}\right)^nn^{n+2k-1}\left(e^{-\sqrt{2k^3/n}}+o(1)\right).
\]
This proves the theorem.
\end{proof}

\section{Conclusion}\label{con}

In a series of papers that started in \cite{FuGiMa1} and was continued in \cite{ChFu,ChFuYu,FuGiMa2,FuHuYu,Ma,YuFuYuZh}, it was shown that the numbers of general networks (${\rm PN}_{n,k}$), tree-child networks (${\rm TC}_{n,k}$), galled networks (${\rm GN}_{n,k}$), and galled tree-child networks (${\rm GTC}_{n,k}$) with $n$ leaves and $k$ reticulations all admit the same asymptotic formula for fixed $k$ as $n$ tends to infinity, namely,
\[
{\rm PN}_{n,k}\sim{\rm TC}_{n,k}\sim{\rm GN}_{n,k}\sim{\rm GTC}_{n,k}\sim\frac{2^{k-1}\sqrt{2}}{k!}\left(\frac{2}{e}\right)^n n^{n+2k-1}.
\]
In fact, the results for tree-child and galled networks follow from those for general networks and galled tree-child networks, respectively, as the former class is a superclass of both, whereas the latter class is contained in both. On the other hand, the classes of tree-child networks and galled networks have a non-trivial intersection. Recently, the results for general networks and tree-child networks were shown to remain valid in the range $k=o(n^{1/2})$, but not beyond; see \cite{YuZh}. This raised the question whether the same is true for the other two classes, too. In this paper, we surprisingly proved that the square-root threshold for general and tree-child networks becomes a cube-root threshold for galled networks and galled tree-child networks.

We proved the above results by first deriving the asymptotic formula for the number of semi-simplex tree-child networks and then extending it to galled tree-child networks and finally to galled networks. Our proof is elementary in the sense that no complex-analytic tools are used. However, we could have alternatively used tools from analytic combinatorics (which are based on complex analysis). We avoided this in order to make the paper also accessible to readers who are not familiar with these tools.

Another result which we proved in this paper, and which in fact is the source of the cube-root phenomenon, is a (sharp) phase transition result for random tree-child networks. More precisely, we proved that for $k=o(n^{1/3})$, a tree-child network with $n$ leaves and $k$ reticulations chosen uniformly at random is almost surely semi-simplex. On the other hand, for $k/n^{1/3}\rightarrow\infty$ and $k=o(n^{1/2})$, the opposite holds, namely, it almost surely is not. Moreover, in the transitional range when $k\sim cn^{1/3}$, we found the limiting probability of the event that a random tree-child network is semi-simplex. To the best of our knowledge, this is the first phase transition result for random phylogenetic networks.

\section*{Acknowledgements} We thank Louxin Zhang for encouragement and helpful comments. MF acknowledges partial support by the National Science and Technology Council (NSTC), Taiwan under grants NSTC-113-2115-M-004-004-MY3. HY is financially supported by the Singapore MOE Academic Research Fund Tier 1 [A-8001951-00-00].

\end{document}